\documentclass[12pt,reqno]{amsart}
\usepackage[T1]{fontenc}
\usepackage{amssymb, a4wide,xy}
\usepackage{amsmath,amsthm, amsfonts}
\usepackage{xcolor}
\xyoption{all}

\usepackage{multicol}
\usepackage{hyperref}

\newtheorem{theorem}[subsection]{Theorem}
\newtheorem{lemma}[subsection]{Lemma}
\newtheorem{proposition}[subsection]{Proposition}
\newtheorem{corollary}[subsection]{Corollary}

\theoremstyle{definition}
\newtheorem{remark}[subsection]{Remark}
\newtheorem{definition}[subsection]{Definition}

\def\pa{\partial}

\def\lam{\lambda}
\newcommand{\La}{\textsf{R}}

\def\lra{\longrightarrow}

\DeclareMathOperator{\Ker}{\mathsf {Ker}}

\DeclareMathOperator{\Img}{\mathsf {Im}}
\DeclareMathOperator{\Z}{\mathsf {Z}}

\def\Der{\operatorname{Der}}
\def\IDer{\operatorname{IDer}}
\def\InnAut{\operatorname{InnAut}}

\def\lim{\operatorname{lim}}

\newcommand{\g}{\mathfrak{g}}
\newcommand{\s}{\mathfrak{s}}
\newcommand{\rr}{\mathfrak{r}}

\newcommand{\h}{\mathfrak{h}}
\newcommand{\f}{\mathfrak{f}}
\newcommand{\m}{\mathfrak{m}}
\newcommand{\n}{\mathfrak{n}}
\newcommand{\cc}{\mathfrak{c}}
\newcommand{\e}{\mathfrak{e}}

\newcommand{\q}{{\textsl{q}}}

\begin{document}

\title[$\q$-crossed modules and $\q$-capability of Lie algebras]{$\q$-crossed modules and $\q$-capability of Lie algebras}

\author[E. Khmaladze]{Emzar Khmaladze$^1$}
\address{$^1$The University of Georgia, Kostava St. 77a, 0171 Tbilisi, Georgia \& A. Razmadze Mathematical Institute of Tbilisi State University,
	Tamarashvili St. 6, 0177 Tbilisi, Georgia   }
\email{e.khmal@gmail.com}
\author[M. Ladra]{Manuel Ladra$^2$}
\address{$^2$Department of Mathematics, CITMAga, Universidade de Santiago de Compostela\\15782 Santiago de Compostela, Spain}
\email{manuel.ladra@usc.es}

\begin{abstract}
Given a non-negative integer $\q$, we study two different notions of the $\q$-capability of Lie algebras via the non-abelian $\q$-exterior product of Lie algebras. The first  is related to the $\q$-crossed modules and inner $\q$-derivations, and the second  is the Lie algebra version of the $\q$-capability of groups proposed by Ellis in 1995.
\end{abstract}



\maketitle

\section{Introduction}\label{S:In}

\subsection{A brief overview of mod $\q$ theories}
 Since the 1980s, many important articles have been published devoted to studying modulo $\q$ (also called finite) versions of algebraic and topological theories. Neisendorfer \cite{Ne} developed the homotopy theory with coefficients in the abelian group $\mathbb{Z}/\q\mathbb{Z}$ (primary homotopy theory) and studied such essential 
  topics as mod $\q$ Hurewicz theorems, exponents of mod $\q$ homotopy groups, mod $\q$ Samelson products and so on.
  Browder \cite{Brd}  investigated the algebraic $K$-theory with $\mathbb{Z}/\q\mathbb{Z}$ coefficients (the mod $\q$ algebraic $K$-theory).
    Suslin and Voevodsky \cite{SuVo} calculated the mod $2$ algebraic $K$-theory of the integers due to Voevodsky's solution of the Milnor conjecture. Karoubi and Lambre \cite{KaLa} constructed the Dennis trace map from mod $\q$ algebraic $K$-theory to mod $\q$ Hochschild homology, having an unexpected relationship with number theory.
    
    A considerable number of papers are devoted to investigating the non-abelian mod $\q$ tensor and exterior products of groups; see, for example, \cite{Brn, CoRo, El2, ElRo}. It is essential to note the applications of these constructions in group homologies with $\mathbb{Z}/\q\mathbb{Z}$ coefficients and in the description of the universal $\q$-central extensions. Furthermore, they are used in developing the mod $\q$ (co)homology theory of groups by Conduch\'e, H. Inassaridze, and N. Inassaridze \cite{CoInIn}.
    
   The mod $\q$ versions of Lie algebras's non-abelian tensor and exterior products \cite{El1, El3} were developed in \cite{Kh1, Kh2}. Among other results,  applications of the mod $\q$ non-abelian tensor and exterior products to the Lie algebra homologies with coefficients in the ground ring modulo $\q$, and to the description of the universal $\q$-central relative extensions of Lie algebras are obtained.
   
\subsection{Shortly on capability and $\q$-capability of groups}
The capability of groups was first studied by  Baer \cite{Ba}, who classified all capable groups among the finitely generated abelian groups. Regarding the classification of prime-power groups, Hall \cite{Ha} found it important and interesting to investigate what condition a group must satisfy to be isomorphic to a central quotient group of another group, or equivalently, to be isomorphic to an inner automorphism group of another group.
This laid the foundation for introducing and investigating the concept of capable groups. We say that \emph{a group $G$ is capable} if there exists a group $H$ such that $G \cong H/\Z(H)$, or equivalently, $G\cong \InnAut(H)$, where $\Z(H)$ denotes the center and $\InnAut(H)$ represents the group of inner automorphisms of  $H$. An overview of the basic theory on the capability of groups is given in \cite{BeTa}. In our interest, it is also essential to mention the article \cite{BeFeSc}, where the notion of the epicenter $\Z^*(G)$ of a group $G$ is introduced and proved that $G$ is capable if and only if  $\Z^*(G)$ is the trivial subgroup of $G$. 

Later on, using the non-abelian exterior square of groups, Ellis introduced the exterior center $\Z^{\wedge}(G)$ of a group $G$ in \cite{El2} and proved that $G$ is capable if and only if $\Z^{\wedge}(G) \cong \{1\}$. In fact, he worked more generally in the modulo q setting, and the result mentioned is the case when $\q = 0$.  More concretely, for any non-negative integer $\q$, he introduced the $\q$-capability of a group and gave its necessary and sufficient condition via the non-abelian $\q$-exterior square introduced in \cite{ElRo}.

Recently, a different notion of a $\q$-capable group was introduced and studied in \cite{BaOlDoRo}, which seems more natural to us than the one in \cite{El2}. The ideas of \cite{BaOlDoRo} inspired us to export similar results for Lie algebras in the modulo $\q$ setting.

\subsection{The content of the paper}
In this paper, we choose to study two different notions of $\q$-capable Lie algebras: the Lie algebraic versions of the respective group-theoretic notions from \cite{BaOlDoRo} and from \cite{El2}. For $\q=0$, both concepts are the same and coincide with the definition of capable Lie algebras introduced in \cite{SaAlMo} and further investigated in \cite{KhKuLa,NiJoPa,NiPaRu}.

The paper is organized as follows. Section~\ref{S:crossed} recalls some basic definitions of $\q$-crossed modules, non-abelian $\q$-tensor and $\q$-exterior products of Lie algebras and their relations
 with the (universal) $\q$-central extensions. In Section~\ref{S:additional}, we study other properties of the non-abelian $\q$-tensor and $\q$-exterior products and give the necessary results for developing the paper.
 In Section~\ref{S:CapL}, we introduce the notion of $\q$-capable Lie algebra (Definition~\ref{Def q-capable}) 
 and give its connection with $\q$-crossed modules and inner $\q$-derivations (Proposition \ref{Prop q-crossed},
  Corollary~\ref{Cor q-der}). Then we introduce the notion of $\q$-exterior center of a Lie algebra  (Definition~\ref{Def q-exterior}) 
  and present one of the paper's main results, proving that  a Lie algebra is $\q$-capable if and only if its $\q$-exterior center is trivial (Theorem~\ref{theorem2}).
   Section~\ref{S:E} introduces one more notion: the notion of the $\q$-tensor center of a Lie algebra (Definition \ref{Def q-tensor center}) and fined a condition under which it coincides with the $\q$-exterior center (Theorem~\ref{T:t_e_centers}).  Finally, in Section~\ref{S:Strongly}, we give the Lie algebra version of the $\q$-capability in the sense of Ellis \cite{El2}, which differs from our notion of $\q$-capability, and we call it strongly $\q$-capability (Definition~\ref{Def strong}). Then we give another main result establishing necessary and sufficient condition for a Lie algebra to be strongly $\q$-capable (Theorem \ref{theoremE1}).

\subsection{Conventions and notation}
Throughout this paper, we denote by $\q$ a non-negative integer and by $\La$ a unital commutative ring unless otherwise
stated. We denote by $\La_{\q}$ the quotient ring $\La/{\q}\La$. We use the term Lie algebra to mean a Lie algebra over $\La$. We denote by $[~,~]$ the Lie bracket.

\section{$\q$-crossed modules of Lie algebras} \label{S:crossed}

Let $\g$ and $\h$ be two Lie algebras. \emph{A Lie (left-)action} of $\g$ on $\h$
is a $\La$-bilinear map $\g \times \h \to \h$, $(g,h)\mapsto {}^gh$ such that the following conditions hold
\begin{itemize}
\item ${}^{[g,g']}h={}^g(^{g'}h)-{}^{g'}(^{g}h)$,
\item ${}^g[h,h']=[{}^gh,h']+[h,{}^gh']$.
\end{itemize}
For example, if $\h$ is an ideal of $\g$, then the Lie bracket in $\g$ yields a Lie action of
$\g$ on $\h$.

\

Let
\[
0\to \h \overset{i}{\lra} \e \overset{\pi}{\lra} \g\to 0
\]
be a central extension of Lie algebras, i.e. $\Ker \pi$ is contained in the center
\[
\Z(\e ) = \big \{  e \in \e \mid  [e, x]=0 \:\text{for each}\: x\in \e \big \}
\]
of the Lie algebra $\e$. Take $g\in \g$ and choose $e_g\in \e$ such that $\pi(e_g)=g$. Then,  for any $h\in \h$ and $e\in \e$ one sets
\[
{}^gh=i^{-1}\big([e_g, i(h)]\big) \quad \text{and} \quad {}^ge=[e_g, e].
\]
These elements are independent on the choice of $e_g$, and we get Lie actions of $\g$ on both $\h$ and $\e$.

\begin{definition}\label{Def_q crossed}
Let $\g$ be a Lie algebra.
\emph{A $\q$-crossed $\g$-module} $(\h, \mu)$ is a homomorphism of Lie algebras $\mu \colon \h \to \g$ with an action  of $\g$ on $\h$ satisfying
\begin{enumerate}
\item[(i)] $\mu({}^gh)=[g, \mu(h)]$,
\item[(ii)] ${}^{\mu(h)}h'=[h,h']$,
\item[(iii)] $\q k=0 $
\end{enumerate}
for all $g\in \g$, $h,h'\in \h$ and $k\in \Ker \mu$.
\end{definition}
Obviously, for $\q=0$, this is just the definition of \emph{a crossed $\g$-module}, see e.g. \cite{KaLo}.
\emph{A morphism of $\q$-crossed $\g$-modules} $\varphi \colon (\h, \mu) \to (\h',\mu')$ is a $\g$-equivariant homomorphism of Lie algebras $\varphi \colon \h \to \h'$, i.e. $\varphi({}^gh)={}^g{\varphi(h)}$ for all $g\in \g$ and $h\in\h$, such that $\mu' \varphi = \mu$.

\

It follows immediately from Definition \ref{Def_q crossed} the following

\begin{lemma}
Let $0\lra \h \overset{i}{\lra} \e \overset{\pi}{\lra} \g\lra 0$
be a $\q$-central extension of Lie algebras, that is, $\Ker \pi$ is contained in the $\q$-center
\[
\Z_{\q}(\e ) = \big \{e \in \Z(\e) \mid \q e = 0 \big \}
\]
 of $\e$ (see \cite{Kh1}). Then $\pi \colon \e \to \g$ with the above-described Lie action of $\g$ on $\e$, is a $\q$-crossed $\g$-module.

Conversely, given a $\q$-crossed $\g$-module $\mu \colon \h\to\g$, it follows from the definition that $\Ker \mu$ is a $\q$-central ideal of $\h$. Thus, if $\mu$ is an epimorphism, then it gives rise to a $\q$-central extension
\[
0\lra \Ker \mu {\lra} \e \overset{\mu}{\lra} \g\lra 0.
\]
\end{lemma}

\

The existence of universal $\q$-central extensions of Lie algebras was studied in \cite{Kh1}, where additionally they are described via the non-abelian exterior product modulo $\q$. This gives us another important example of $\q$-crossed modules of Lie algebras and will be presented below in this section.

%
%

First, we need to recall that, given a Lie algebra $\g$ and two crossed $\g$-modules $\m \to \g$ and $\n \to \g$,
the non-abelian $\q$-tensor and $\q$-exterior products, $\m \otimes ^{\q} \n$ and
$\n \wedge^{\q} \m$ are defined and studied in \cite{Kh1, Kh2}.
In this paper, we need just to use these notions for the identity $1_{\g} \colon \g {\to} \g$ and the inclusion $\h \hookrightarrow \g$ crossed $\g$-modules, where $\h$ is an ideal of $\g$.

\begin{definition} (\cite{Kh1, Kh2}) \emph{The non-abelian $\q$-tensor product} of a Lie algebra $\g$ and its ideal $\h$, for $\q\geq 1$,
is the Lie algebra $\h \otimes ^{\q} \g$ generated by symbols $h\otimes g$ and  $\{h\}$
with $g\in \g$, $h\in \h$ subject to the following relations for $h, h' \in \mathfrak h$, $g, g' \in \mathfrak g$ and $\lam, \lam' \in \La$:
\begin{align}
&\lam (h\otimes g) = \lam h\otimes g = h\otimes \lam g , \label{t1}
\\
&(h+h')\otimes g = h\otimes g + h'\otimes g , \label{t2}
\\
&h\otimes (g+g') = h\otimes g + h\otimes g' ,\label{t3}
\\
&[h, h']\otimes g = h \otimes [h', g] - h' \otimes [h, g] ,\label{t4}
\\
&h \otimes [g, g'] = [g', h]\otimes g - [g, h] \otimes  g',\label{t5}
\\
&[h\otimes g , h'\otimes g' ] = [h, g] \otimes [h', g'],\label{t6}
\\
&[\{h'\}, h\otimes g]=[\q h', h]\otimes g + h\otimes [\q h', g],\label{t7}
\\
&\{\lam h + \lam' h' \} =  \lam \{ h\} +  \lam' \{ h' \},\label{t8}
\\
&[\{h\}, \{h'\}] = \q h \otimes \q h' ,\label{t9}
\\
&\{[h, g]\} = \q (h\otimes g).\label{t10}
\end{align}

For $\q = 0$, we define $\h \otimes^0 g$ to be the non-abelian tensor product $\h \otimes \g$ of Lie algebras introduced by Ellis in \cite{El1} (see also \cite{InKhLa}), that is, $\h \otimes^0 \g$ is the Lie algebra generated by the symbols $h \otimes g$ for $h\in \h$, $g\in \g$ and subject to relations \eqref{t1}--\eqref{t6}.

Further, \emph{the non-abelian $\q$-exterior product}, $\h \wedge ^{\q} \g$, $\q \geq 0$, of $\g$ and its ideal $\h$ is the quotient of the Lie algebra $\h \otimes ^{\q} \g$ by the relation
\begin{equation}
h\otimes h = 0, \quad h\in \h. \label{t11}
\end{equation}
\end{definition}

Let us denote by $h\wedge g$ the coset of $h\otimes g$ in $\mathfrak h \wedge ^q \mathfrak g$.
There is a natural epimorphism of Lie algebras $\mathfrak h \otimes ^q \mathfrak g \to \mathfrak h \wedge ^q \mathfrak g$ given by
$h\otimes g \mapsto h\wedge g$ and $\{h\} \mapsto \{h\}$.

\begin{lemma} \label{lemma xi}
Let $\h$ be an ideal of a Lie algebra $\g$.
\begin{enumerate}
\item[i)] There is a Lie homomorphism
$
\xi^{\otimes} \colon \h \otimes^{\q}\g \to \g
$
defined on generators by
\[
\xi^{\otimes}(h\otimes g)=[h,g] \quad \text{and} \quad \xi^{\otimes}(\{h\})=\q h.
\]
 Moreover, $\xi^{\otimes}$ factors through the non-abelian $\q$-exterior product $\h \wedge^{\q}\g$, and there is an induced Lie homomorphism
 $ \xi^{\wedge} \colon \h \wedge^{\q}\g \to \g $.
\item[ii)] If $x\in \h \otimes^{\q}\g$ (resp. $x\in \h \wedge^{\q}\g$) then $\{\xi^{\otimes}(x)\}=\q x$ (resp. $\{\xi^{\wedge}(x)\}=\q x$).
\end{enumerate}
\end{lemma}
\begin{proof}
i) It is readily seen that $\xi ^{\otimes}$ commutes with all defining relations \ref{t1}--\eqref{t11}.

\noindent ii)  First, note that $\Img \xi^{\otimes} \subseteq \h$ (resp. $\Img \xi^{\wedge} \subseteq \h$).

For the generators $h\otimes g, \{h\} \in \h \otimes^{\q}\g$, we have
\begin{align*}
&\{\xi^{\otimes}(h\otimes g)\} = \{[h,g]\}=\q (h\otimes g)  \quad \text{by \eqref{t10}},\\
&\{\xi^{\otimes}({\{h\}})\}=\{\q h\}=\q\{h\} \quad \text{by \eqref{t8}}.
\end{align*}
Now, if $x\in \h \otimes^{\q}\g$ (resp. $x\in \h \wedge^{\q}\g$ ) is any element,  the defining relations \eqref{t6}--\eqref{t9} guarantee that the identity
$\{\xi^{\otimes}(x)\}=\q x$ (resp. $\{\xi^{\wedge}(x)\}=\q x$) holds.
\end{proof}

\begin{proposition}\label{Prop 2.6}
There is a Lie action of $\g$ on $\h \otimes^{\q}\g$ (resp. $\h \wedge^{\q}\g$) given by
\begin{align*}
&{}^{g'}(h\otimes g)=[g', h]\otimes g + h\otimes [g',g], \quad {}^g\{h\}=\{[g,h]\}\\
(\text{resp.} \quad &{}^{g'}(h\wedge g)=[g', h]\wedge g + h\otimes [g',g], \quad {}^g\{h\}=\{[g,h]\}).
\end{align*}
Moreover, with this action the homomorphism $\xi^{\otimes} \colon \h \otimes^{\q}\g \to \g$ (resp. $\xi^{\wedge} \colon \h \wedge^{\q}\g \to \g$) is a $\q$-crossed $\g$-module.
\end{proposition}
\begin{proof}
The verification that the given equations define the Lie action  of $\g$ on $\h \otimes^{\q}\g$  (resp. $\g$ on $\h \wedge^{\q}\g$) and that with this action $\xi^{\otimes}$ (resp. $\xi^{\wedge}$) is a crossed $\g$-module, follows as a particular case from \cite[Proposition 1.14]{Kh1} (resp. \cite[Corollary 1.15]{Kh1}). At the same time, for all $x\in \Ker \xi^{\otimes}$ (resp. $x\in \Ker \xi^{\wedge}$ ), by Lemma \ref{lemma xi} (ii), we get
\begin{align*}
&\q x =\{\xi^{\otimes}(x)\}=\{0\}=0\\
( \text{resp.} \quad &\q x =\{\xi^{\wedge}(x)\}=\{0\}=0 ).
\end{align*}
This completes the proof.
\end{proof}

Let us note that $\Img \xi^{\otimes}=\Img \xi^{\wedge}$ and it is a submodule of $\h$ generated by the elements $[h, g]$ and $\q h'$ for $h, h'\in \h$,
$g \in \g$. We denote it by $\h \#_{\q} \g$. In fact, $\h \#_{\q} \g$ is an ideal of $\g$.
In particular, $\g\#_{\q} \g$ is an ideal of $\g$.

Recall that $\g$ is called \emph{$\q$-perfect Lie algebra}, if $\g = \g\#_{\q} \g$. The following result is a particular case of \cite[Theorem 2.8]{Kh1}.
\begin{remark}
If $\g$ is a $\q$-perfect Lie algebra, then $\g \otimes^{\q}\g=\g \wedge^{\q}\g$ and the $\q$-crossed module $\xi^{\otimes} \colon \g \otimes^{\q}\g \to \g$ is the universal $\q$-central extension of $\g$.
\end{remark}

\section{Further properties of the non-abelian $q$-tensor and $q$-exterior products}\label{S:additional}

\begin{proposition}\label{abelian1} Let $\g$ be an abelian Lie algebra, i.e. $[\g, \g]=0$.
Then $\g \otimes^{\q} \g $ (resp. $\mathfrak g \wedge^q \mathfrak g $) is an abelian Lie algebra and there is an isomorphism:
\[
\g \otimes^{\q}\g \cong  \g \times
\big((\g/ \q\g) \otimes_{\La_{\q}}(\g/\q\g)\big)
\]
\[
(\text{resp.} \quad \g \wedge^{\q}\g \cong  \g \times
\big((\g/{\q}\g) \wedge_{\La_{\q}}(\g/{\q}\g)\big) \ ).
\]

\end{proposition}\label{abelian2}
\begin{proof} From the defining relations, it is easy to see that $\g \otimes^{\q} \g $ is abelian. Moreover,
\[
\q(h \otimes  g)=\{[h, g]\}=\{0\}=0,
\]
for each $g, h\in \mathfrak g$. Therefore, there is the following diagram of (abelian) Lie algebras

\begin{equation*}
\xymatrix @=20mm {
	\g \ar@<-1.1ex>[r]_{\sigma_1 \ \ } & \g \otimes^{\q}\g \ar@<-1.1ex>[l]_{\pi_1 \ \  } \ar@<1.1ex>[r]^{\pi_2 \ \ \ \ \ \ } & (\g/{\q}\mathfrak g) \otimes_{\La_{\q}}(\g/{\q}\g)\ar@<1.1ex>[l]^{\sigma_2 \ \ \ \ \ \ },
}
\end{equation*}
where
\begin{align*}
&\sigma_1(g)= \{g\};\\
&\pi_1\{g\} = g, \ \pi_1(g\otimes h) = 0;\\
&\sigma_2\big((g+\q\g)\otimes (h+\q\g)\big) = g \otimes h;\\
&\pi_2( g\otimes h) = (g+\q\g)\otimes (h+\q\g), \ \pi_2\{g\}= 0.
\end{align*}
One easily checks that $\pi_1\sigma_1 =1_{\g}$,
$\pi_2\sigma_2 = 1_{ (\g/{\q}\g) \otimes_{\La_{\q}}(\g/{\q}\g)}$ and
$\pi_1\sigma_1 +\pi_2\sigma_2 =1_{\g \otimes^{\q} \g}$.
Hence the first isomorphism holds. The proof is essentially the same for the second isomorphism.
\end{proof}

\begin{proposition}\label{prop1} Given a Lie algebra $\g$ and its ideal $\h$, the canonical sequence of Lie algebras
\[
\h \wedge^{\q} \g \lra \g \wedge^{\q} \g \lra (\g / \h) \wedge^{\q} (\g/ \h) \lra 0 ,
\]
is exact.
\end{proposition}
\begin{proof} For $\q=0$, this exact sequence was obtained in \cite{El1}, and the proof for any $\q$ is similar.
\end{proof}

We have the following lemma by combining  \cite[Proposition 1.2]{El3} and \cite[Lemma 2.3]{Kh2}.
\begin{lemma}\label{lemma1} If $\q=0$, or $\q\geq 1$, and $\La$ is a $\q$-torsion-free ground ring, then for any free Lie algebra $\f$, the homomorphism $\xi^{\wedge} \colon \f \wedge^{\q} \f \to \f$ induces an isomorphism
$\f \wedge^{\q} \f \cong \f \#_{\q} \f$.
\end{lemma}


To prove one of our main results - Theorem \ref{T:t_e_centers}, we will use the description of the kernel of the canonical projection $\g \otimes^{\q} \g \twoheadrightarrow \g \wedge^{\q} \g$ via Whitehead's universal quadratic functor  $\Gamma (-)$. Let us recall from \cite{SiTy} that it is defined for any $\La$-module $A$ to be the $\La$-module $\Gamma (A)$ generated by symbols $\gamma (a)$, for each $a\in A$, subject to the relations
\begin{align*}
&\lam^2\gamma (a)=\gamma (\lam a), \\
&\gamma (a+b+c) + \gamma (a) + \gamma (b) + \gamma (c) = \gamma (a+b) + \gamma (a+c) + \gamma (b+c), \\
&\gamma (\lam a+b) + \lam\gamma (a) + \lam\gamma (b) = \lam\gamma (a+b) + \gamma (\lam a) + \gamma (b),
\end{align*}
for each $\lam \in \La$, $a, b, c \in A$.

\begin{theorem}\label{theorem1} There is a well-defined homomorphism of $\La$-modules
\[
i \colon \Gamma \big ( \g / (\g \#_{\q} \g)\big) \to \g \otimes^{\q} \g, \quad \text{given by} \quad i\big(\gamma (g+ \g \#_{\q} \g) \big) = g\otimes g,
\]
 and the following conditions hold:
\begin{enumerate}
\item[(i)] The sequence
\[
\Gamma \big ( \g / (\g \#_{\q} \g)\big) \overset{i}{\longrightarrow}
 \g\otimes^{\q} \g \longrightarrow
 \g\wedge^{\q} \g  \longrightarrow 0
\]
is exact;

\item[(ii)] Consider $\Gamma \big ( \g / (\g \#_{\q} \g)\big)$ as an abelian Lie algebra.
Then $i$ is a homomorphism of Lie algebras;

\item[(iii)] If $\g / (\g \#_{\q} \g)$ is a free $\La_{\q}$-module, then $i$ is a monomorphism.
\end{enumerate}

\end{theorem}
\begin{proof} (i) Let $\g \otimes \g$ denote the non-abelian tensor square of $\g$ \cite{El1}.
Then, there is a well-defined homomorphism $\Gamma \big ( \g / [\g, \g]\big) \to \g \otimes  \g$,
$\gamma (g + [\g, \g]) \mapsto g\otimes g$, \cite{El1}. Therefore, to show that $i$ also is well defined, it suffices
to check that $i\big(\gamma (g + \q h)\big) = i\big(\gamma (g)\big)$ for each $g, h \in \g$. Indeed,
\begin{align*}
i\big(\gamma (g + q h)\big) &= (g + q h) \otimes (g + q h) \\
& =g \otimes g + q (g\otimes h) + q (h \otimes g) + q^2 (h \otimes h) \\
&=g \otimes g + \{[g, h]\} + \{[h, g]\} + q \{[h, h]\} \\
& =g \otimes g + \{[g, h] + [h, g]+ q[h, h]\}=g \otimes g = i\big(\gamma(g)\big) .
\end{align*}

(ii) We must show that $i\Big (\Gamma \big ( \g / (\g \#_{\q} \g)\big) \Big)$ is a Lie algebra with a
trivial Lie bracket. Indeed,
\[
[i(\gamma (g)), i(\gamma (h))]=[g\otimes g, h\otimes h] = [g, g]\otimes [h, h] =0,
\]
for each $g, h \in \g$.

(iii) Consider the following commutative diagram
\begin{equation*}
\xymatrix@+10pt{
 \Gamma\big(\g/(\g\#_{\q}\g)\big) \ \ar@{->}[r]
\ar@{=}[d]_{ }
 &\g \otimes ^{\q} \g
\ar@{->}[d]
 \\
  \Gamma\big(\g/(\g\#_{\q}\g)\big) \ar@{->}[r]
  &\big(\g/(\g\#_{\q}\g)\big)\otimes^{\q}\big(\g/(\g\#_{\q}\g)\big) ,
}\end{equation*}
where the right vertical arrow is the natural homomorphism. By Proposition \ref{abelian1}, we have
\[
\big(\g/(\g\#_{\q}\g)\big)\otimes^{\q} \big(\g/(\g\#_{\q}\g)\big)\cong
\big(\g/(\g\#_{\q}\g)\big) \times \Big(\big(\g/(\g\#_{\q}\g)\big)
\otimes_{\La_{\q}} \big(\g/(\g\#_{\q}\g)\big) \Big).
\]
Therefore, since $\g / (\g \#_{\q} \g)$ is a free $\La_{\q}$-module, the lower horizontal arrow is a
monomorphism \cite[Proposition 17]{El1}. Hence, the upper horizontal arrow is a monomorphism too.
\end{proof}

\




Given an epimorphism of Lie algebras $\pa \colon \f \twoheadrightarrow \g$, one easily observes that the induced homomorphism
\[
\bar{\pa} \colon \f/ (\Ker(\pa) \#_{\q} \f)  \twoheadrightarrow \g
\]
is a ${\q}$-central extension of Lie algebras. In particular, if $\pa \colon \f \twoheadrightarrow \g$
is a free presentation of $\g$, i.e. $\f$ is a free Lie algebra,  then
$\bar\pa \colon \f / (\Ker(\pa) \#_{\q} \f)\twoheadrightarrow \g$
will be called \emph{the ${\q}$-central extension induced by this free presentation}.

\begin{lemma}\label{lemma2}
Suppose that the ground ring $\La$ is a $\q$-torsion-free ring.

Let
 $0\longrightarrow \rr \longrightarrow \f \overset{\pa}{\longrightarrow} \g \longrightarrow 0$ be a free presentation of a Lie algebra $\g$.
Denote $\cc=\f/(\rr\#_{\q} \f)$. Let $\bar{\pa} \colon \cc \to \g$ be the $\q$-central extension induced by the given free
 presentation of $\g$ . Then there is a natural isomorphism of Lie algebras
\[
\cc\#_{\q} \cc\overset{\cong}{\longrightarrow} \g \wedge ^{\q} \g, \quad \text{given by} \quad
[c_1, c_2]\mapsto \bar{\pa}(c_1)\wedge \bar{\pa}(c_2)\quad \text{and} \quad \q c \mapsto \{ \bar{\pa}(c)\}.
\]
\end{lemma}
\begin{proof} We have an isomorphism
$\cc\#_{\q}\cc\cong (\f\#_{\q} \f)/(\rr\#_{\q} \f)$.
By Proposition \ref{prop1} we get the exact sequence of Lie algebras
 $\rr\wedge^{\q}\f\overset{i}{\longrightarrow} \f\wedge^{\q}\f\longrightarrow
\g\wedge^{\q}\g\longrightarrow 0$ and so
$(\f\wedge^{\q}\f)/{i(\rr\wedge^{\q}\f)}\cong \g\wedge^{\q}\g$.
On the other hand, since  $\f$ is a free Lie algebra, by Lemma \ref{lemma1},  we have the isomorphism
$\xi^\wedge \colon \f\wedge^{\q}\f\overset{\cong}{\longrightarrow} \f\#_{\q} \f$
 and $\xi^\wedge i (\rr\wedge^{\q}\f)= \rr\#_{\q} \f$. Then we get
\[
\cc\#_{\q}\cc\cong (\f\#_{\q} \f)/(\rr\#_{\q} \f)\cong
(\f\wedge^{\q}\f)/{i(\rr\wedge^{\q}\f)}\cong \g\wedge^{\q}\g .
\]
\end{proof}

\section{$\q$-capability of Lie algebras}\label{S:CapL}

 In complete analogy to the notion of capable Lie algebras, we formulate the following definition of $\q$-capable Lie algebras.
\begin{definition}\label{Def q-capable}
A Lie algebra $\g$ is called \emph{$\q$-capable} if there exists a Lie algebra $\m$ such that
\[
\g \cong \m / \Z_{\q}(\m).
\]
\end{definition}

It is well-known that the notion of capable Lie algebra is related to inner derivations. Namely, a Lie algebra is capable if and only if it is isomorphic to the Lie algebra of inner derivations of a Lie algebra. Below we examine a similar result for the modulo $\q$ setting. Before, we need to recall some notations.

Let $\m$ be a Lie algebra. \emph{A derivation} of $\m$ is a linear self-map $d \colon \m \to \m$ such that $d[m,m']=[d(m),m']+[m,d(m')]$ for all $m,m'\in \m$. For example, for any $m\in \m$, the map $d_m \colon \m\to \m$, given by $d_m(m')=[m,m']$ is a derivation, called \emph{the inner derivation} defined by $m\in \m$.

Denote by $\Der(\m)$ the set of all derivations of $\m$. It is a Lie algebra with the bracket $[d,d']=dd' - d'd$. The set of all inner derivations $\IDer(\m)=\{d_m \mid m\in \m \}$ is an ideal of $\Der(\m)$, and it is called \emph{the Lie algebra of inner derivations}.

Given a Lie algebra $\m$, let us consider the Lie algebra structure on $\Der(\m)\times \q\m$ given by
\[
[(d, \q m), (d', \q m')]=\big([d,d'], \q[m,m']\big).
\]
It is easy to check that there is a Lie action of $\Der(\m)\times \q\m$ on $\m$ defined by
\[
{}^{(d, \q m)}m' = d (m').
\]
There is a Lie homomorphism
\[
\pa \colon \m \lra \Der(\m)\times \q\m, \quad \pa(m) = (d_m, \q m),
\]
the image of which is the Lie subalgebra
\[
\IDer(\m, \q)=\{(d_m, \q m) \mid m\in \m \}
\]
of $\Der(\m)\times \q\m$,  and the kernel is the $\q$-center $\Z_{\q}(\m)$ of $\m$. We call $\IDer(\m, \q)$ \emph{the Lie algebra of inner $\q$-derivations}.

\begin{proposition}\label{Prop q-crossed}
Given a Lie algebra $\m$, there is a short exact sequence of Lie algebras
\[
0\lra \Z_{q}({\m}) \lra \m \overset{\pa}{\lra} \IDer(\m, \q) \lra 0,
\]
 where $\pa$ is a $\q$-crossed $\IDer(\m, \q)$-module with the Lie action of  $\IDer(\m, \q)$ on $\m$ given by
\begin{equation}\label{action inner}
{}^{(d_m, \q m)}m' = [m,m'].
\end{equation}
\end{proposition}
\begin{proof}
It is easy to check that the sequence is a short exact sequence of Lie algebras, and \eqref{action inner} indeed defines a Lie action.
 To show that $\pa$ is a $\q$-crossed module, we only need to verify the first two conditions of Definition \ref{Def_q crossed}.

Thus, we will show that $\pa(^{(d_m, \q m)}m')=[(d_m, m), \pa(m')]$. In fact, we have
 \begin{align*}
 \pa(^{(d_m, \q m)}m')&=\pa([m,m'])= (d_{[m,m']}, \q [m,m'])=([d_m,d_{m'}], \q [m,m']) \\ & = [(d_m, \q m), (d_{m'}, \q m')]  = [(d_m, m), \pa(m')].
  \end{align*}
 Here we have used the obvious equality $d_{[m,m']}=[d_m,d_{m'}]$ for inner derivations.
And finally, we see that
 \begin{align*}
 {}^{\pa({m})} m'={}^{(d_m, \q m)}m'=[m,m'].
  \end{align*}

\end{proof}

\begin{corollary}\label{Cor q-der}
A Lie algebra is $\q$-capable if and only if it is isomorphic to the Lie algebra of inner $\q$-derivations $\IDer(\m, \q)$ for some Lie algebra $\m$.
\end{corollary}

\begin{definition} \label{Def q-exterior} Let $\g$ be a Lie algebra. The $\q$-\emph{exterior center} of $\g$ is defined by
\[
\Z^{\wedge}_{\q}(\g ) = \Big \{  g \in \g \mid  \{g\}=0_{\g \wedge^{\q} \g} \:\text{and} \: g\wedge x=0_{\g \wedge^{\q} \g}
\:\text{for each}\: x\in \g   \Big \}.
\]
\end{definition}

\begin{lemma}\label{cor1} Let $\bar{\pa} \colon \cc \twoheadrightarrow \g$ be a $\q$-central extension induced by a free
 presentation of $\g$. Then $c\in \Z_{\q}(\cc)$ if and only if $\bar{\pa}(c)\in \Z^{\wedge}_{\q}(\g)$.
\end{lemma}
\begin{proof} If $c\in \Z_{\q}(\cc)$,  then ${\q} c =0$ and $[c, x]=0$ for each $x\in \cc$. By Lemma~\ref{lemma2},
$\bar{\pa}(c) \wedge \bar{\pa}(x) =0_{\g \wedge^{\q} \g}$ and
$\{\bar{\pa}(c)\}=0_{\g \wedge^{\q} \g}$. Since $\bar{\pa}$ is an epimorphism,
$\bar{\pa}(c)\in Z^{\wedge}_{\q}(\g)$.

Conversely, if $\bar{\pa}(c)\in  \Z^{\wedge}_{\q}(\g)$, then $[c, x]=0$ for each $x\in \g$ because the isomorphic
image of $[c, x]$ is $0_{\g \wedge^{\q} \g}$. Moreover
$\{\bar{\pa}(c)\}=0_{\g \wedge^{\q} \g}$. Hence, by the same argument $\q c=0$. Thus,  $c\in \Z_{\q}(\cc)$.
\end{proof}

\begin{theorem}\label{theorem2}
A Lie algebra $\g$ is $\q$-capable if and only if $\Z^{\wedge}_{\q}(\g)=0$.
\end{theorem}
\begin{proof}
First, suppose that $\Z^{\wedge}_{\q}(\g)=0$. Consider a $\q$-central extension $\bar{\pa} \colon \cc \twoheadrightarrow \g$
 induced by a free presentation of $\g$. We claim that $\Ker \bar{\pa} =\Z_{\q}(\cc)$. Indeed, since
 $\bar{\pa} \colon \cc \twoheadrightarrow \g$ is a $\q$-central extension, we have $\Ker \bar{\pa} \subseteq \Z_{\q}(\cc)$.
To see the inverse inclusion, we note that given any $x\in \Z_{\q}(\cc)$, by Lemma  \ref{cor1}
$\bar{\pa}(x)\in \Z_{\q}^{\wedge}(\g)=0$ and so $x\in \Ker \bar{\pa}$. Hence, $\g \cong \cc / \Z_{\q}(\cc)$, i.e. $\g$ is capable.

Conversely, let $\g$ be a $\q$-capable Lie algebra. Then there is a Lie algebra $\m$ such that
$\m/\Z_{\q}(\m)\cong \mathfrak{g}$. In other words, there is an epimorphism of Lie algebras
$\pi \colon \m\twoheadrightarrow \g$ such that $\Ker \pi =\Z_{\q}(\m)$.  Consider a free presentation
$0\to \s \to \f \overset{\tau}{\to} \m\to 0$ of $\m$. Then we get a free presentation
$0\to \rr\to \f \overset{\pa}{\to} \g\to 0$ of $\g$, where $\pa= \pi\tau$ and
$\rr$ denotes the kernel of $\pa$. All these data give the following commutative diagram with exact rows

\begin{equation*}
\xymatrix@+10pt{
0\ \ar@{->}[r]& \rr \ \ar@{->}[r]
\ar@{->}[d]_{ }
 &\f\ar@{->}[r]^{\pa}
\ar@{->}[d]_{\tau }
&\g \ar@{->}[r]
\ar@{=}[d]_{ }
&0 \\
0\ \ar@{->}[r]
& \Z_{\q}(\m) \ar@{->}[r]
 & \m \ar@{->}[r]^{\pi}
& \g \ar@{->}[r]
&0 .
}\end{equation*}
Now we claim that $\tau (\rr\#_{\q} \f)=0$. Indeed, given any $r\in \rr$ and $f\in \f$,
we have $\tau([r,f])=[\tau(r), \tau(f)]=0$ because $\tau(r)\in \Z_{\q}(\m) =\Ker\pi$. Moreover, for each $r\in \rr$,
$\tau (\q r)=\q\tau(r)=0$ because $\tau(r)\in \Z_{\q}(\m)$. Denote $\f/(\rr\#_{\q} \f)$
by $\cc$. Then we have the following induced commutative triangle of Lie algebras

 \begin{equation*}
\xymatrix@+10pt{
& \cc\ar@{->}[rd]^{\bar{\pa}}
\ar@{->}[rr]^{\bar{\tau} }
& &\m \ar@{->}[ld]_{\pi}
 \\
& &   \g  & 
}
\end{equation*}
Since $\bar{\tau}$ is an epimorphism, $\bar{\tau}(\Z_{\q}(\cc))\subseteq \Z_{\q}(\m)= \Ker \pi $, therefore
$\bar{\pa}(\Z_{\q}(\cc))=0$, i.e. $\Z_{\q}(\cc)\subseteq \Ker \bar{\pa}$. On the other hand, since
$\bar{\pa} \colon \cc\to \g$ is a $\q$-central extension, $\Ker \bar{\pa} \subseteq \Z_{\q}(\cc)$.
We get that $\Z_{\q}(\cc)= \Ker \bar{\pa}$. Then, by Lemma \ref{cor1} we have
\[
\Z^{\wedge}_{\q}(\g)=\bar{\pa}(\Z_{\q}(\cc))=\bar{\pa}(\Ker \bar{\pa})=0.
\]
\end{proof}

\section{The coincidence of $\q$-exterior and $\q$-tensor centers}\label{S:E}

In this section, we introduce the tensor analog of the $\q$-exterior center of a Lie algebra and prove their coincidence under additional conditions. The idea comes from the work \cite{Ni2}, which later was used to establish a similar result for Leibniz algebras in \cite{KhKuLa}.

\begin{definition}\label{Def q-tensor center}  The $\q$-\emph{tensor center} of a Lie algebra $\g$ is defined by
\[
\Z^{\otimes}_{\q}(\g ) = \Big \{  g \in \g \; | \{g\}=0_{\g \otimes^{\q} \g} \:\text{and}\: g\otimes x=0_{\g \otimes^{\q} \g}
\:\text{for each}\: x\in \g   \Big \}.
\]
\end{definition}

Let us note that, using Lemma \ref{lemma xi}, we have obvious inclusions
\[
\Z^{\otimes}_{\q}(\g ) \subseteq \Z^{\wedge}_{\q}(\g ) \subseteq \Z_{\q}(\g ) \subseteq \Z(\g ).
\]

\begin{lemma} \label{lemma3}
Let $i  \colon \Gamma \big ( \g / (\g \#_{\q} \g)\big) \to \g \otimes ^{\q} \g$
be the homomorphism as in Theorem \ref{theorem1} and $p \colon \g \otimes ^{\q} \g \to \big ( \g / (\g \#_{\q} \g)\big) \otimes^{\q} \big ( \g / (\g \#_{\q} \g)\big)$ be the canonical projection.
If $\g / (\g \#_{\q} \g)$ is a free $\La_{\q}$-module, then
\[
i\Big(\Gamma \big ( \g / (\g \#_{\q} \g)\big)\Big) \cap \Ker p
= \{0\}.
\]
\end{lemma}
\begin{proof} Consider the following commutative diagram:
\begin{equation*}
\xymatrix@+10pt{
 \Gamma\big(\g/(\g\#_{\q}\g)\big) \ \ar@{->}[r]^i
\ar@{=}[d]_{ }
 &\g \otimes ^{\q} \g
\ar@{->}[d]^p
 \\
  \Gamma\big(\g/(\g\#_{\q}\g)\big) \ar@{->}[r]
  &\big(\g/(\g\#_{\q}\g)\big)\otimes^{\q}\big(\g/(\g\#_{\q}\g)\big) .
}\end{equation*}
By Theorem \ref{theorem1}, both horizontal arrows in this diagram are monomorphisms, and the assertion follows.
\end{proof}

\begin{proposition}\label{P:pr}
If $\g/(\g\#_{\q} \g)$ is a free $\La_{\q}$-module, then there is an isomorphism of Lie algebras
\begin{equation}\label{E:eq2}
\g\otimes^{\q} \g\cong (\g \wedge^{\q}\g)\times
 \Gamma\big(\g/(\g\#_{\q}\g)\big).
\end{equation}
\end{proposition}
\begin{proof} By Theorem \ref{theorem1}, we have the following commutative diagram of Lie algebras with exact rows:
\begin{equation*}
\xymatrix@+10pt{
0\ \ar@{->}[r]& \Gamma\big(\g/(\g\#_{\q} \g)\big) \ \ar@{->}[r]^{}
 \ar@{=}[d]_{ }
 &\g\otimes^{\q} \g\ar@{->}[r]
\ar@{->}[d]
&\g\wedge^{\q} \g\ar@{->}[r]
 \ar@{->}[d]_{ }
&0 \\
0\ \ar@{->}[r]
& \Gamma\big(\g/(\g\#_{\q} \g)\big) \ar@{->}[r]^{}
&\frac{\g}{\g\#_{\q} \g}\otimes^{\q} \frac{\g}{\g\#_{\q} \g}
\ar@{->}[r]
& \frac{\g}{\g\#_{\q} \g}\wedge^{\q} \frac{\g}{\g\#_{\q} \g}
\ar@{->}[r]
&0.
}\end{equation*}
The lower row is the exact sequence of abelian Lie algebras, i.e. just $\La_{\q}$-modules. Moreover, by Proposition
\ref{abelian1}
\[
\big(\g/(\g\#_{\q}\g)\big)\wedge^{\q} \big(\g/(\g\#_{\q}\g)\big)\cong
\big(\g/(\g\#_{\q}\g)\big) \times \Big(\big(\g/(\g\#_{\q}\g)\big)
\wedge_{\La_{\q}} \big(\g/(\g\#_{\q}\g)\big)\Big).
\]
Since $\g/(\g\#_{\q}\g)$ is a free $\La_{\q}$-module,
$\big(\g/(\g\#_{\q}\g)\big)\wedge_{\La_{\q}}
\big(\g/(\g\#_{\q}\g)\big)$ is also a free $\La_{\q}$-module. Therefore, the lower row in the
diagram splits. Consequently, the upper row splits too.
\end{proof}

\begin{theorem}\label{T:t_e_centers} If $\g/(\g\#_{\q} \g)$ is a free $\La_{\q}$-module, then
\[
Z^{\otimes}_{\q}(\mathfrak{g})=Z^{\wedge}_{\q}(\mathfrak{g}).
\]
\end{theorem}
\begin{proof}
First, let us prove that
\[
\Z^{\wedge}_{\q}(\g)\cap (\g\#_{\q}\g) =\Z^{\otimes}_{\q}(\g).
\]
For that, we take $g\in \Z^{\wedge}_{\q}(\g)\cap (\g\#_{\q}\g)$. Then  $g\in \Z^{\wedge}_{\q}(\g)$
implies that
\[
g\otimes x,  \{g\}\in \Ker \big(\g \otimes^{\q} \g {\lra} \g \wedge^{\q} \g\big)
\]
for all $x\in \g$.
And $g\in \g\#_{\q}\g$ implies that
\[
g\otimes x,  \{g\}\in \Ker \Big(\g \otimes^{\q} \g \lra
\big(\g/  (\g\#_{\q}\g)\big) \wedge^{\q} \big(\g /  (\g\#_{\q}\g)\big)\Big),
\]
for each $x\in \g$. By Lemma \ref{lemma3}, $g\otimes x=0_{\g \otimes^{\q} \g}$ and
$\{g\}=0_{\g \otimes^{\q} \g}$, i.e. $g\in \Z^{\otimes}_{\q}(\g)$.

To show the converse inclusion $\Z^{\otimes}_{\q}(\g)\subseteq\Z^{\wedge}_{\q}(\g)\cap (\g\#_{\q}\g)$, it suffices to prove that $\Z^{\otimes}_{\q}(\g)\subseteq \g\#_{\q} \g$.
Take $g\in \Z^{\otimes}_{\q}(\g)$. Then $\{g\}=0_{\g \otimes^{\q} \g}$. Hence,
$\{\bar{g}\}=0_{(\g/ (\g\#_{\q}\g)) \otimes^{\q} (\g/  (\g\#_{\q}\g) )}$,
where $\bar{g}=g+\g\#_{\q}\g$. Since $\g/(\g\#_{\q} \g)$ is a free
$\La_{\q}$-module, by Proposition \ref{abelian1}
\[
\big(\g/(\g\#_{\q}\g)\big)\otimes^{\q} \big(\g/(\g\#_{\q}\g)\big)\cong
\big(\g/(\g\#_{\q}\g)\big) \times \Big(\big(\g/(\g\#_{\q}\g)\big)
\otimes_{\La_{\q}} \big(\g/(\g\#_{\q}\g)\big)\Big).
\]
Using this isomorphism we deduce that
$\{\bar{g}\}=0_{(\g/  (\g\#_{\q}\g)) \otimes^{\q} (\g/  (\g\#_{\q}\g) )}$
if and only if $\bar{g}=0_{\g/  (\g\#_{\q}\g)}$. Thus, $g\in \g\#_{\q} \g$.

Finally, by using  Proposition \ref{abelian1} again (for the non-abelian $\q$-exterior product), one can repeat the reasoning given in the previous paragraph to show that $\Z^{\wedge}_{\q}(\g)\subseteq
\g\#_{\q} \g$. This completes the proof.
\end{proof}

\

\section{Strongly $\q$-capable Lie algebras}\label{S:Strongly}


Our definition of a $\q$-capable Lie algebra for $\q=0$ gives the notion of a capable Lie algebra. So, it is a modulo $\q$ generalization of the usual capability of Lie algebras.

In the case of groups, according to Ellis' definition in \cite{El2}, a group $G$ is $\q$-capable if there exists a group $M$ whose center and $\q$-center coincide and the quotient of $M$ by its center ($\q$-center) is isomorphic to $G$.
Now we give the Lie algebra version of the $\q$-capability in the sense of Ellis, which differs from our notion of $\q$-capability.

\begin{definition}\label{Def strong} We say that a Lie algebra $\g$ is \emph{strongly $\q$-capable} (or $\q$-\emph{capable} in the sense of Ellis) if there exists a Lie algebra
$\m$ such that $\Z(\m) = \Z_{\q}(\m)$ and $\g \cong \m/\Z(\m)$.
\end{definition}

Note that any strongly $\q$-capable Lie algebra is $\q$-capable at the same time.

In the rest of this section, we describe a necessary and sufficient condition of strongly $\q$-capability of a Lie algebra via the non-abelian $\q$-exterior product.

Let $\g$ be a Lie algebra, and $\h$ be an ideal of $\g$.
Let $\h \otimes \g$ (resp. $\h \wedge \g$) denote the non-abelian tensor (resp. exterior)
product of $\h$ and $\g$ (see \cite{El1}). We have the natural homomorphisms
$\h \otimes \g \to \h \otimes^{\q} \g$ and
$\h \wedge \g \to \h \wedge^{\q} \g$. We introduce the following notations:
\begin{align*}
&\h \boxtimes \g = \Img \big(\h \otimes \g \lra \h \otimes^{\q} \g \big),\\
&\h \curlywedge^{\q} \g = \Img \big(\h \wedge \g \lra \h \wedge^{\q} \g \big).
\end{align*}
The proof of the following proposition is standard.
\begin{proposition}\label{propE1} Let $0 \to \mathfrak h \to \mathfrak g \to \mathfrak g' \to 0$ be an extension of Lie algebras.
Then we have the following exact sequence
\[
\h \curlywedge^{\q} \g \lra \g \curlywedge^{\q} \g \lra
\g' \curlywedge^{\q} \g' \lra 0 .
\]
\end{proposition}

In what follows, the ground ring $\La$ is assumed to be $\q$-torsion-free.

\begin{lemma} \label{lemmaE1} Let $\f$ be a free Lie algebra.
Assume that $\xi^{\wedge} \colon \f \wedge ^{\q} \f \to \f$ is defined as in Lemma~\ref{lemma xi}. Then
$\xi^{\wedge}$ induces an isomorphism $\f \curlywedge^{\q}\f \cong [\f, \f]$.
\end{lemma}
\begin{proof} Straightforward from Lemma \ref{lemma1}.
\end{proof}

\begin{lemma}\label{lemmaE2}
Let $\g$ be a Lie algebra and $\bar{\pa} \colon \cc \twoheadrightarrow \g$ be the $\q$-central extension induced by a free
 presentation $0\longrightarrow \rr\longrightarrow \f \overset{\pa}{\longrightarrow} \g\longrightarrow 0$ of $\g$. Then there is a natural isomorphism of Lie algebras
\[
[\cc, \cc]\overset{\cong}{\longrightarrow} \g \curlywedge^{\q} \g,
\quad [c_1, c_2]\mapsto \bar{\pa}(c_1)\wedge \bar{\pa}(c_2) .
\]
\end{lemma}
\begin{proof} Straightforward from Lemma \ref{lemma2}.
\end{proof}

\begin{definition} The  $\q$-\emph{tensor} and $\q$-\emph{exterior centers} in the sense of Ellis of a Lie algebra $\g$ are defined by
\begin{align*}
&\Z^{\boxtimes}_{\q}(\g ) = \Big \{  g \in \g \mid g\otimes x=0_{\g \otimes^{\q} \g}
\:\text{for each}\: x\in \g \Big \},\\
&\Z^{\curlywedge}_{\q}(\g) = \Big \{ g \in \g \mid g\wedge x=0_{\g \wedge^{\q} \g}
\:\text{for each}\: x\in \g \Big \}.
\end{align*}
\end{definition}

\begin{remark} For arbitrary Lie algebra $\g$, we have that
$\Z^{\otimes}_{\q}(\g) \subseteq \Z^{\boxtimes}_{\q}(\g )$ and
$\Z^{\wedge}_{\q}(\g) \subseteq \Z^{\curlywedge}_{\q}(\g)$. In general, the inverse inclusions are not true. For instance, if $\g =\La$ is considered as an abelian Lie algebra, then $\Z^{\curlywedge}_{\q}(\La) =\La$. Whilst $\Z^{\wedge}_{\q}(\La) =\{0\}$, since $\La$ is a $\q$-torsion-free ring.
\end{remark}

\begin{corollary}\label{corE1} Let $\bar{\pa} \colon \cc \twoheadrightarrow \g$ be a $\q$-central extension induced by a free
 presentation of $\g$. Then $c\in \Z(\cc)$ if and only if $\bar{\pa}(c)\in \Z^{\curlywedge}_{\q}(\g)$.
\end{corollary}
\begin{proof} Straightforward from Lemma \ref{lemmaE2}.
\end{proof}


\begin{theorem}\label{theoremE1} A Lie algebra $\g$ is strongly $\q$-capable if and only if
$\Z^{\curlywedge}_{\q}(\g) =\{0\}$.
\end{theorem}
\begin{proof} The proof is essentially the same as that of Theorem \ref{theorem2}. But still, we are going to give full details
for the reader's convenience.

First, suppose that $\Z^{\curlywedge}_{\q}(\g)=0$. Consider a $\q$-central extension
$\bar{\pa} \colon \cc \twoheadrightarrow \g$ induced by a free presentation of $\g$. It suffices to show that $\Ker \bar{\pa} =\Z_{\q}(\cc)=\Z(\cc)$.
Since $\bar{\pa} \colon \cc \twoheadrightarrow \g$ is a $\q$-central extension, we have
$\Ker \bar{\pa} \subseteq \Z_{\q}(\cc) \subseteq \Z(\cc)$. To prove the inverse inclusions it suffices to show that
$\Z(\cc) \subseteq \Ker \bar{\pa}$. Indeed, if $x\in \Z(\cc)$, then by Corollary~\ref{corE1},
$\bar{\pa}(x)\in \Z^{\curlywedge}_{\q}(\g)=\{0\}$ and so $x\in \Ker \bar{\pa}$.

Conversely, let $\g$ be a strongly $\q$-capable Lie algebra. Then there is a Lie algebra $\m$
such that $\Z_{\q}(\m)=\Z_{\q}(\m)$ and $\m/\Z(\m)\cong \g$. Hence,
there is an epimorphism of Lie algebras $\pi \colon \m\twoheadrightarrow \g$ such that
$\Ker \pi =\Z(\m)=\Z_{\q}(\m)$.
Let $0\to \s\to \f \overset{\tau}{\to} \m\to 0$ be a free presentation of $\m$.
Consider the free presentation $0\to \rr\to \f \overset{\pa}{\to} \g\to 0$ of $\g$,
where $\pa= \pi\tau$ and  $\rr$ denotes the kernel of $\pa$. Then we have  the following commutative diagram
 with exact rows

\begin{equation*}
\xymatrix@+10pt{
0\ \ar@{->}[r]& \rr \ \ar@{->}[r]
\ar@{->}[d]_{ }
 &\f\ar@{->}[r]^{\pa}
\ar@{->}[d]_{\tau }
&\g \ar@{->}[r]
\ar@{=}[d]_{ }
&0 \\
0\ \ar@{->}[r]
& \Z(\m) = \Z_{\q}(\m) \ar@{->}[r]
 & \m \ar@{->}[r]^{\pi}
& \g \ar@{->}[r]
&0 .}
\end{equation*}
We have $\tau (\rr\#_{\q} \f)=0$. Indeed, for each $r\in \rr$ and $f\in \f$,
$\tau[r,f]=[\tau(r), \tau(f)]=0$ because $\tau(r)\in \Z(\m) =\Ker\pi$. Moreover, for each $r\in \rr$,
$\tau (\q r)=\q\tau(r)=0$ because $\tau(r)\in \Z_{\q}(\m)$. Denote $\f/(\rr\#_{\q} \f)$
by $\cc$. Then we have the following induced commutative triangle of Lie algebras
\begin{equation*}
\xymatrix@+10pt{
& \cc\ar@{->}[rd]^{\bar{\pa}}
\ar@{->}[rr]^{\bar{\tau} }
& &\m \ar@{->}[ld]_{\pi}
 \\
& &   \g &  
}
\end{equation*}
Since $\bar{\tau}$ is an epimorphism, $\bar{\tau}(\Z(\cc))\subseteq \Z(\m)= \Ker \pi $, therefore
$\bar{\pa}(\Z(\cc))=0$, i.e. $\Z(\cc)\subseteq \Ker \bar{\pa}$. On the other hand, since
$\bar{\pa} \colon \cc\to \g$ is a $\q$-central extension,
$\Ker \bar{\pa} \subseteq \Z_{\q}(\cc) \subseteq \Z(\cc)$.
We get that $\Z(\cc)= \Ker \bar{\pa}$. Then, by Corollary \ref{corE1} we have
\[
\Z^{\curlywedge}_{\q}(\g)=\bar{\pa}(\Z(\cc))=\bar{\pa}(\Ker \bar{\pa})=0.
\]
\end{proof}


\begin{proposition} If $\g$ is a perfect Lie algebra, i.e. $\g = [\g, \g]$, then
\[
\Z^{\boxtimes}_{\q}(\g) = \Z^{\curlywedge}_{\q}(\g) = \Z(\g).
\]
\end{proposition}

So, if $\La$ is a $\q$-torsion-free ring, then $\g=\La$ is a $\q$-capable but not a strongly $\q$-capable Lie algebra.
\begin{proof} We have well-defined homomorphisms
$\g\boxtimes^{\q} \g \to \g\curlywedge^{\q}\g \to \g$, given by $g\otimes h \mapsto g\wedge h$ and  $g\wedge h\mapsto [g, h]$, respectively. They imply the
inclusions
\[
\Z^{\boxtimes}_{\q}(\g)  \subseteq \Z^{\curlywedge}_{\q}(\g) \subseteq \Z(\g).
\]
Therefore, it suffices to show that $\Z(\g) \subseteq \Z^{\boxtimes}_{\q}(\g)$.
Let $x\in \Z(\g)$ and $g\in \g$. Then there are $g_i, h_i \in \g$ and $\lam_i \in \La$,
for $i=1,\dots, n$, such that $g=\overset{n}{\underset{i=1}{\sum}}\lam_i [g_i, h_i]$. Then
\[
x \otimes g = x\otimes \overset{n}{\underset{i=1}{\sum}}\lam_i [g_i, h_i] =
\overset{n}{\underset{i=1}{\sum}} \lam_i \big(x\otimes  [g_i, h_i]\big) =
\overset{n}{\underset{i=1}{\sum}} \lam_i \big([h_i, x]\otimes g_i -  [g_i, x]\otimes h_i \big) = 0,
\]
which means that $x\in \Z^{\boxtimes}_{\q}(\g)$ and the proof is completed.
\end{proof}

\begin{corollary} A perfect Lie algebra $\g$ is $\q$-capable if and only if $\Z(\g)=0$.
\end{corollary}

\

\subsection*{Acknowledgements}
We thank Guram Donadze for the useful discussions. 
\subsection*{Funding}
The authors were supported by Agencia Estatal de Investigación de Espa\~na. 
(Spain, European FEDER support included), grant PID2020-115155GB-I00. Emzar Khmaladze was financially supported by EU fellowships for Georgian researchers, 2023 (57655523) and by  Shota Rustaveli National Science Foundation of Georgia, grant FR-22-199.

\end{document}